\documentclass[12pt]{amsart}
\usepackage{latexsym, amsmath, amssymb, mathrsfs}
\usepackage{hyperref} 

\theoremstyle{plain}
  \newtheorem{Thm}{Theorem}[section] 
  \newtheorem{Lma}[Thm]{Lemma} 
  \newtheorem{Cor}[Thm]{Corollary} 
  \newtheorem{Prop}[Thm]{Proposition}

\theoremstyle{definition}
  \newtheorem{Def}[Thm]{Definition}

  \newtheorem{Dis}[Thm]{Discussion}

\theoremstyle{remark}
  \newtheorem{Rem}[Thm]{Remark}

\newcommand{\mlabel}[1]%
  {\mbox{}\marginpar{\raggedleft\hspace{0pt}{\rm\ttfamily#1}}\label{#1}}

\newcommand{\rank}{\operatorname{rank}}

\newcommand{\R}[1]{{R^{({#1})}}}
\newcommand{\M}[1]{{M^{({#1})}}}

\newcommand{\fm}{{\mathfrak m}}

\newcommand{\ringR}{(R,\fm,k)} 

\newcommand{\cx}{\operatorname{cx}}

\newcommand{\Abs}[1]{\Vert #1 \Vert}

\newcommand{\sC}{\mathscr{C}}
\newcommand{\sR}{\mathscr{R}}
\newcommand{\sF}{\mathscr{F}}

\newcommand{\sV}{\mathscr{V}}

\newcommand{\segre}{\,\sharp\,} 
\newcommand{\ol}{\overline}

\newcommand{\floor}[1]{\left\lfloor{#1}\right\rfloor}
\newcommand{\trun}{\vert}
\providecommand{\card}[1]{\left\lVert{#1}\right\rVert}
\providecommand{\abs}[1]{\left\lvert{#1}\right\rvert}
\providecommand{\rank}{\operatorname{rank}}

\renewcommand{\ge}{\geqslant} \renewcommand{\le}{\leqslant} 
\renewcommand{\geq}{\geqslant} \renewcommand{\leq}{\leqslant} 

\newcounter{hours}\newcounter{minutes}

\newcommand{\excise}[1]{}
\begin{document}

\title[On the Frobenius complexity of determinantal rings]
{\bf On the Frobenius complexity of determinantal rings}
\author[Florian~Enescu]{Florian Enescu}
\author[Yongwei~Yao]{Yongwei Yao}
\address{Department of Mathematics and Statistics, Georgia State
  University, Atlanta, GA 30303 USA} 
\email{fenescu@gsu.edu}
\email{yyao@gsu.edu}

\date{}

\begin{abstract}
We compute the Frobenius complexity for the determinantal ring of
prime characteristic $p$ obtained by
modding out the $2 \times 2$ minors of an $m \times n$ matrix of
indeterminates, where $m > n \ge 2$. 
We also show that, as $p \to \infty$, the Frobenius complexity
approaches $m-1$.
\end{abstract}

\maketitle

\section{Introduction}

\subsection{Notations} Throughout this paper $R$ is a commutative Noetherian ring, often local, of prime characteristic $p$. Let $q=p^e$, where $e \in \mathbb{N}  = \{ 0, 1, \ldots \}$. Consider the $e$th Frobenius homomorphism $F^e:R\to R$ defined $F(r)=r^q$, for all $r \in R$. For an $R$-module $M$, an $e$th Frobenius action (or Frobenius operator) on $M$ is an additive map $\phi: M \to M$ such that $\phi(rm)= r^{p^e}\phi(m)$, for all $r \in R, m \in M$. For any $e \geq 0$, we let $R^{(e)}$ be the $R$-algebra defined as
follows: as a ring $\R{e}$ equals $R$ while the $R$-algebra structure
is defined by $r \cdot s = r^{q} s$, for all $r \in R,\, s \in
R^{(e)}$.  Also,
$\R{e}$ as an $\R{e}$-algebra is simply $R$ as an $R$-algebra. Similarly, for an $R$-module $M$, we can define
a new $R$-module structure on $M$ by letting $r * m = r^{p^e}m$, for all $r\in R, m\in M$. We denote this
$R$-module by $\M{e}$.  

Consider now an $e$th Frobenius action, $\phi: M \to M$, on $M$, which
is no other than an $R$-module homomorphism $\phi : M \to \M{e}$. Such
an action naturally defines an $R$-module homomorphism $f_{\phi}:
\R{e} \otimes_R M \to M$, where $f_{\phi} (r \otimes m) = r \phi(m)$,
for all $r \in R, m \in M$. Here, $\R{e}$ has the usual structure
(i.e., without twisting) as
an $R$-module given by $\R{e}=R$ on the left, while on the right we
have the twisted module structure via the Frobenius action.

Let $\sF^e(M)$ be the collection of all $e$th Frobenius operators on
$M$. The $R$-module structure on $\sF^e{(M)}$ is given by viewing
$\M{e}$ as an $R$-module without twisting, that is,
$(r\phi)(x) = r\phi(x)$ for every $r \in R,\, \phi  \in \sF^e{(M)}$
and $x \in M$. 

\begin{Def}\label{def:sfe} 
We define {\it the algebra of Frobenius operators} on $M$ by
$$\sF{(M)} = \oplus_{e\geq 0} \sF^e{(M)},$$
with the multiplication on $\sF{(M)}$ determined by composition of
functions; that is, if $\phi \in \sF^e{(M)}, \psi \in \sF^{e'}{(M)}$
then $\phi \psi:=\phi\circ \psi \in \sF^{e+e'}(M)$. Hence, in
general, $\phi \psi \neq \psi \phi$.
\end{Def}

Note that $\sF^0{(M)} = {\rm End}_R(M)$, which is a subring of
$\sF(M)$. Naturally, each $\sF^e{(M)}$ is a module over
$\sF^0{(M)}$. Since $R$ maps canonically to $\sF^0{(M)}$, this makes
$\sF^e{(M)}$ an $R$-module by restriction of scalars. Note that $(\phi
\circ r)(m) = \phi(rm) = (r^q\phi)(m)$, for all $r \in R, m \in
M$. Therefore, $\phi r = r^q \phi$, for all $r\in R, \, \phi \in
\sF^e{(M)}$, $q=p^e$.

\subsection{The Frobenius Complexity}

The main concept studied in this paper is the Frobenius complexity of
a local ring $R$, which was introduced in~\cite{EY}. In fact, the
results in this subsection, if not referenced otherwise, are taken
from \cite{EY}. We first need to review the 
definition of the complexity of a graded ring. 

\begin{Def}Let $A= \oplus_{e\geq 0} A_e$ be a $\mathbb{N}$-graded ring, not
  necessarily commutative. 
\begin{enumerate}
\item Let $G_e(A)= G_e$ be the subring of $A$ generated by the elements of
  degree less or equal to $e$. (So $k_0 = 0$.) 
We agree that $G_{-1} = A_0$.

\item We use $k_e=k_e(A)$ to denote the minimal number of homogeneous
  generators of $G_e$ as a subring of $A$ over $A_0$. We say that $A$
  is {\it degree-wise finitely generated} if $k_e < \infty$ for all
  $e$. We agree that $k_{-1} = 0$.

\item For a degree-wise finitely generated ring $A$, we say that a set
  $X$ of homogeneous elements of $A$ minimally generates $A$ if for all
  $e$, $X_{\leq e} =\{ a \in X: deg(a) \leq e \}$ is a minimal set of
  generators for $G_e$ with $k_e = |X_{\leq e} |$ for every $e \ge 0$. Also,
  let $X_e= \{ a \in X: deg(a)=e \}$. 
\end{enumerate}
\end{Def}

\begin{Prop} \label{prop:minimal-gen}
With the notations introduced above,  
let $X$ be a set of homogeneous elements of $A$. Then
\begin{enumerate}
\item The set $X$ generates $A$ as a ring over $A_0$ if and only if 
$X_{\le e}$ generates $G_e$ as a ring over $A_0$ for all $e \ge 0$ 
if and only if the image of $X_{e}$ generates
$\frac{A_{e}}{(G_{e-1})_{e}}$ as an $A_0$-bimodule for all $e \ge 0$.  
\item Assume that $A$ is degree-wise finitely generated
  $\mathbb{N}$-graded ring and $X$ generates $A$ as a ring over $A_0$. 
The set $X$ minimally generates $A$ as a ring over $A_0$ if and
  only if $|X_{e}|$ is the minimal number of generators (out of all
 homogeneous generating sets) of $\frac{A_{e}}{(G_{e-1})_{e}}$ as an
  $A_0$-bimodule for all $e \ge 0$.  
\end{enumerate}
\end{Prop}

\begin{Cor} \label{cor:minimal-gen}
Let $A$ be a degree-wise finitely generated $\mathbb{N}$-graded ring
and $X$ a set of homogeneous elements of $A$. Then
\begin{enumerate}
\item The minimal number of generators of $\frac{A_{e}}{(G_{e-1})_{e}}$ as an
  $A_0$-bimodule is $k_e - k_{e-1}$ for all $e \ge 0$. 
\item If $X$ is generates $A$ as a ring over $A_0$ then 
$|X_e| \ge k_e - k_{e-1}$ for all $e \ge 0$. 
\end{enumerate}
\end{Cor}

\begin{Def}
\label{degreewise}
Let $A$ be a degree-wise finitely generated ring.
The sequence $\{k_e\}_e$ is called the {\it growth} sequence for $A$. The {\it complexity} sequence is given by $\{ c_{e}(A)= k_{e}-k_{e-1} \}_{e\geq 0}$.  
The {\it complexity} of $A$ is
$$\inf \{ n \in \mathbb{R}_{> 0}: c_e(A)=k_{e} - k_{e-1} = O(n^e) \}$$ 
and it is denoted by $\cx(A)$. If there is no $n >0$ such that $c_e(A)= O(n^e)$, then we say that $\cx(A)=\infty$.
\end{Def}

\begin{Def}\label{def:nearly-onto}
Let $A$ and $B$ be $\mathbb N$-graded rings and $h \colon A \to B$ be
a graded ring homomorphism. We say that $h$ is \emph{nearly onto} if 
$B = B_0[h(A)]$ (that is, $B$ as a ring is generated by $h(A)$ over $B_0$). 
\end{Def}

\begin{Thm}
\label{thm:nearly-onto}
Let $A$ and $B$ be $\mathbb N$-graded rings that are degree-wise
finitely generated. If there
exists a graded ring homomorphism $h \colon A \to B$ that
is nearly onto, then $c_e(A) \ge c_e(B)$ for all $e \ge 0$.
\end{Thm}

\begin{Def}
Let $A$ be a $\mathbb{N}$-graded ring such that there exists a ring homomorphism $R \to A_0$, where $R$ is a commutative ring. We say that $A$ is a (left) $R$-{\it skew algebra} if $aR \subseteq Ra$ for all homogeneous elements $a \in A$. A right $R$-skew algebra can be defined analogously. In this paper, our $R$-skew algebras will be left $R$-skew algebras and therefore we will drop the adjective `left' when referring it to them.
\end{Def}

\begin{Cor}
\label{interpretation}
Let $A$ be a degree-wise finitely generated $R$-skew algebra such that
$R=A_0$. Then $c_{e}(A)$ equals the minimal number of generators of
$\frac{A_{e}}{(G_{e-1})_{e}}$ as a left $R$-module for all $e$.
\end{Cor}

We are now in position to state the definition of the Frobenius
complexity of a local ring of prime characteristic.

\begin{Def} \label{def-FCX}
Let $\ringR$ be a local ring of prime characteristic $p$. 
We define the {\it Frobenius complexity} of the ring $R$ by 
$$\cx_F(R) = \log _p (\cx(\sF{(E)})).$$
Also, denote $k_e(R) : = k_e (\sF{(E)})$, for all $e$, and call these
numbers the {\it Frobenius growth sequence} of $R$. Then $c_e= c_e(R):
= k_{e}(R)-k_{e-1}(R)$ defines the {\it Frobenius complexity sequence}
of $R$. If the Frobenius growth sequence of the ring $R$ is eventually
constant (i.e., $\cx(\sF{(E)}) = 0$), then the Frobenius complexity of
$R$ is set to be $-\infty$. 
If $\cx(\sF(E)) = \infty$, the Frobenius complexity if $R$ is set to
be $\infty$. 
\end{Def}

Katzman, Schwede, Singh and Zhang have introduced an important $\mathbb{N}$-graded ring in their paper~\cite{KSSZ}, which is an example of an $R$-skew algebra. We will study the complexity of this skew-algebra in this section, and apply these results to the complexity of the ring $R$ in subsequent sections.

\begin{Def}[\cite{KSSZ}]
Let $\sR$ be an $\mathbb{N}$-graded commutative ring of prime
characteristic $p$ with $\sR_0 =R$.  
Define $T(\sR):= \oplus_{e\geq 0} \sR _{p^e-1}$, which is an
$\mathbb{N}$-graded ring by  
$$ a *b = ab^{p^e}$$
for all $a \in\sR _{p^e-1} ,\, b\in \sR _{p^{e'}-1}$.
The degree $e$ piece of $T(\sR)$ is $T_e(\sR)=\sR _{p^e-1}$. 
\end{Def}

A number of results have been proved about the Frobenius complexity of
a local ring and they are summarized below. 

\begin{Thm}[\cite{EY}, Corollary 2.12, Theorems 4.7, 4.9]
Let $\ringR$ be a local ring. 
\begin{enumerate}
\item
If $R$ is $0$-dimensional then $\cx_F(R) = -\infty$.
\item
If $R$ is normal, complete and has dimension at most two, then
$\cx_F(R) \leq 0$. 
\item
If $R$ is normal, complete and has a finitely geneated anticanonical
cover, then $\cx_F(R) <\infty$. 
\end{enumerate}
\end{Thm}

In addition the following holds.

\begin{Thm}[\cite{KSSZ} Proposition 4.1 and \cite{EY} Theorem 4.5]
If $\ringR$ is normal and $\mathbb{Q}$-Gorenstein, then the order of its canonical module in the divisor class group is relatively prime to $p$ if and only if $\cx_F(R) =-\infty$.
\end{Thm}

As in \cite{EY}, we will also use the following notations and
terminologies in the sequel: 
For an integer $a \in \mathbb N$, if $a= c_{n}p^{n} + \cdots + c_1p +
c_0$ with $0 \leq c_i \leq p-1$ for all $0 \le i \le n$,
then we use $a = \overline{c_{n} \cdots c_0}$ to denote the base $p$
expression of $a$. 
Also, we write $a \trun_e$ to denote the remainder of $a$
when dividing to $p^e$. Thus, if $a = \ol{c_{n} \cdots c_0}$
then $a \trun_e = \overline{c_{e-1} \cdots c_0}$, 
which we refer to as the $e$th truncation of $a$.
Put differently, $a \trun_e = a-\floor{\frac a{p^e}}p^e$, in which
$\floor{\frac a{p^e}}$ is the floor function of $\frac a{p^e}$.
When adding up integers $a_i \in \mathbb N$ with $1 \le i \le m$, all
written in base $p$ expressions, we can talk about the carry over to
digit corresponding to $p^e$, which is simply 
$\floor{\frac {a_1 \trun_e + \dotsb + a_m\trun_e}{p^e}}$. 
These notations depend on the choice of $p$, which should be clear
from the context.

For any positive integers $p$ and $m$ (with $p$ prime), denote by
$M_{p,m}(i)$ (or simply $M(i)$ if $p$ and $m$ are understood) the rank
of  $(R[x_1,\,\dotsc,\, x_m]/(x_1^p,\,\dotsc,\, x_m^p))_i$ over $R$, for
all $i \in \mathbb Z$. This is clearly independent of $R$. 
Observe that $M_{p,m} =0$ exactly when $i > d(p-1)$ or $i < 0$. 
In fact, all $M_{p,m}(i)$ can be read off from the following Poincar\'e
series (actually a polynomial):
\[
\sum_{i = -\infty}^{\infty} M_{p,m}(i)t^i = \left(\frac{1-t^p}{1-t}\right)^m
=\left(1+\dotsb + t^{p-1}\right)^m.
\]

\subsection{Determinantal rings} \label{sec:det}

In this paper we consider the determinantal ring $K[X]/I$ where 
$X$ is an $m \times n$ matrix of indeterminates and $I$ is the ideal
of all the $2 \times 2$ minors of $X$ and $K$ a field.
This ring is isomorphic to the Segre product of 
$K[x_1,\dotsc, x_{m}]$ and $K[y_1,\dotsc, y_{n}]$. 

Recall that, for $\mathbb N$-graded commutative rings 
$A = \oplus_{i \in \mathbb N}A_i$ and $B = \oplus_{i \in \mathbb N}B_i$
such that $A_0 = R = B_0$, their Segre product is 
\[
A \segre B = \oplus_{i \in \mathbb N}(A_i \otimes_R B_i),
\]
which is a ring under the natural operations.

\begin{Def}
\label{def:segre-complete}
Let $S_{m,n}$ denote the completion of 
$K[x_1,\dotsc, x_{m}] \segre K[y_1, \dotsc, y_{n}]$ 
with respect to the ideal generated by all homogeneous elements of
positive degree, in which $K$ is a field and $m > n \ge 2$.  
It is easy to see that
\begin{align*}
S_{m,n} &\cong \prod_{\alpha \in \mathbb N^{m},\, 
\beta \in \mathbb N^{n},\, |\alpha| = |\beta|} 
Kx^{\alpha}y^{\beta} \\
&=\left\{
\sum_{|\alpha| = |\beta|}a_{\alpha,\, \beta}x^{\alpha}y^{\beta} 
\; \Big | \; 
a_{\alpha,\, \beta} \in K, \, \alpha \in \mathbb N^{m},\, 
\beta \in \mathbb N^{n}\right\}
\subset K[[x_1, \dotsc,x_{m}, y_1, \dotsc, y_{n}]]. 
\end{align*}
Let $\mathscr R_{m,n}$ be the anticanonical cover of $S_{m,n}$.
\end{Def}

The anticanonical cover of such a ring was described by Kei-ichi Watanabe.

\begin{Thm} [{\cite[page 430]{Wa}}] \label{thm:Watanabe}
Let $K$ be a field and $m > n \ge 2$. 
The anticanonical cover of the Segre product of $K[x_1, \dotsc,
x_{m}]$ and $K[y_1, \dotsc, y_{n}]$ is isomorphic to 
\[
\bigoplus_{i \in \mathbb N} 
\left(\bigoplus_{\alpha \in \mathbb N^{m},\, 
\beta \in \mathbb N^{n},\, |\alpha| - |\beta|=i(m-n)} 
Kx^{\alpha}y^{\beta}\right),
\]
in which the grading is governed by $i$. 
Here, for $\alpha = (a_1, \dotsc, a_m)$ and $\beta = (b_1, \dotsc,
b_{n})$ we denote $x^{\alpha} = x_1^{a_1} \dotsm x_m^{a_m}$ and
$y^{\beta} = y_1^{b_1} \dotsm y_{n}^{b_{n}}$.
\end{Thm}

It follows from Theorem~\ref{thm:Watanabe} that 
\[
\mathscr R_{m,n} \cong \bigoplus_{i \in \mathbb N} 
\left(\prod_{\alpha \in \mathbb N^{m},\, 
\beta \in \mathbb N^{n},\, |\alpha| - |\beta|=i(m-n)} 
Kx^{\alpha}y^{\beta}\right),
\]
in which the grading is governed by $i$.

\begin{Lma}[\cite{EY}]
\label{lma:nearly-onto}
Let $A$ and $B$ be degree-wise finitely generated $\mathbb N$-graded
commutative rings and   
$h \colon A \to B$ be a graded ring homomorphism.
\begin{enumerate}
\item The homomorphism $h$ is nearly onto if and only if $B_i$
is generated by $h(A_i)$ as a $B_0$-module for all $i \in \mathbb N$
(that is, $B$ is generated by $h(A)$ as a $B_0$-module).
\item If $A$ and $B$ have prime characteristic $p$ and $h$
is nearly onto, then the induced graded homomorphism  
$T(h) \colon T(A) \to T(B)$ is nearly onto.
\end{enumerate}
\end{Lma}

\begin{Cor}\label{cor:nearly-onto}
Let $A$ and $B$ be $\mathbb N$-graded commutative rings of prime
characteristic $p$. 
If there exists a graded ring homomorphism $h \colon A \to B$ that
is nearly onto, then $c_e(T(A)) \ge c_e(T(B))$ for all $e \ge 0$.
\end{Cor}

\begin{Prop}[{Compare with \cite[Proposition~5.5]{EY}}] 
\label{prop:nearly-onto}
Let $K$, $S_{m,n}$ and $\sR_{m,n}$ be as in Definition~\ref{def:segre-complete}
with $m > n \ge 2$. 
Then there are nearly onto graded ring homomorphisms from $\sR_{m,n}$ to
$V_{m-n}(K[x_1, \dotsc, x_{m}])$ and vice versa, in which
$V_{m-n}(K[x_1, \dotsc, x_{m}])$ denotes the $(m-n)$-Veronese subring
of $K[x_1, \dotsc, x_{m}]$.
\end{Prop}

\begin{proof}
In light of Definition~\ref{def:segre-complete} and
Theorem~\ref{thm:Watanabe}, we simply assume  
\begin{align*}
\mathscr R_{m,n} = \bigoplus_{i \in \mathbb N} 
\left(\prod_{\alpha \in \mathbb N^{m},\, 
\beta \in \mathbb N^{n},\, |\alpha| - |\beta|=i(m-n)} 
Kx^{\alpha}y^{\beta}\right).
\end{align*}
Define $\phi \colon \sR_{m,n} \to V_{m-n}(K[x_1, \dotsc, x_{m}])$ 
and $\psi \colon V_{m-n}(K[x_1, \dotsc, x_{m}]) \to \sR_{m,n}$ by 
\begin{align*}
\phi(f(x_1,\dotsc, x_m,\, y_1, \dotsc, y_{n}))
&= f(x_1,\dotsc, x_m,\, 0, \dotsc, 0) \in K[x_1, \dotsc, x_{m}]\\
\text{and} \qquad 
\psi(g(x_1,\dotsc, x_m)) &= g(x_1,\dotsc, x_m) \in \sR_{m,n},
\end{align*}
for all $f(x_1,\dotsc, x_m,\, y_1, \dotsc, y_{n}) \in \sR_{m,n}$ 
and all $g(x_1,\dotsc, x_m) \in V_{m-n}(K[x_1, \dotsc, x_{m}])$.

It is routine to verify that both $\phi$ and $\psi$ are graded ring
homomorphisms. 
As $\phi \circ \psi$ is the identity map, we see that $\phi$ is onto
and hence nearly onto.
Finally, note that for every $i \in \mathbb N$, $(\sR_{m,n})_i$ is
generated by $\psi(V_{m-n}(K[ x_1, \dotsc, x_{m}])_i) 
= \psi(K[ x_1, \dotsc, x_{m}]_{i(m-n)})$ as a module over
$(\sR_{m,n})_0 = S_{m,n}$. So $\psi$ is nearly onto, completing the proof.
\end{proof}

\begin{Thm}\label{thm:segre}
Let $K$, $S_{m,n}$ and $\sR_{m,n}$ be as in 
Definition~\ref{def:segre-complete}
with $m > n \ge 2$. 

\begin{enumerate}
\item Then $\sR_{m,n}$ and $V_{m-n}(K[x_1,\dotsc, x_{m}])$ have the same
complexity sequence.
\item If $K$ has prime characteristic $p$, then $T(\sR_{m,n})$ and
  $T(V_{m-n}(K[x_1,\dotsc, x_{m}]))$ have the same complexity sequence. 
\item If $K$ has prime characteristic $p$, then
\[
\cx(\sF(E_{m,n})) = \cx(T(\sR_{m,n})) = \cx(T(V_{m-n}(K[x_1,\dotsc, x_{m}]))), 
\]
in which $E_{m,n}$ stands for the injective hull of the residue field
of $S_{m,n}$. Consequently, 
\[
\cx_F(S_{m,n}) = \log_p\cx(T(V_{m-n}(K[x_1,\dotsc, x_{m}]))).
\]
\end{enumerate}
\end{Thm}

\begin{proof}
This follows from Corollary~\ref{cor:nearly-onto},
Proposition~\ref{prop:nearly-onto} and \cite[Theorem~3.3]{KSSZ}.
\end{proof} 

In summary, to compute the Frobenius complexity of $S_{m,n}$ with
$m > n \ge 2$, it suffices to study $T(V_{r}(K[x_1,\dotsc, x_{m}]))$
with $r = m-n$ (hence $0< r \le m-2$).
The next section is devoted to the study of $T(V_{r}(K[x_1,\dotsc,
x_{m}]))$, more generally with $1 \le r,\,m \in \mathbb N$.

\section{Investigating $T(V_r(R[x_1,\,\dotsc,\,x_m]))$}

Let $R$ be a commutative ring of prime characteristic $p$ and $r,\,m$
positive integers. 
In this section, we study $T(V_r(R[x_1,\,\dotsc,\,x_m]))$. In
particular, we are interested in when it is finitely
generated over $R$, as well as how to compute its complexity. 

To simplify notation, denote the following (with $R$, $p$, $m$ and $r$
understood): 
\begin{itemize}
\item $\sR := R[x_1,\,\dotsc,\,x_m]$. 
\item $\sV := V_r(\sR) = V_r(R[x_1,\,\dotsc,\,x_m])$.
\item $T: = T(\sV) =T(V_r(R[x_1,\,\dotsc,\,x_m]))$.
\item $G_e:= G_{e}(T)$.
\item $T_e:=T_e(\sV) = T_e(V_r(R[x_1,\,\dotsc,\,x_m])) =
  \sR_{r(p^e-1)} = (R[x_1,\,\dotsc,\,x_m])_{r(p^e-1)}$.  
As there are several gradings going on, when we say the degree of a
monomial, we agree that it refers to its (total) degree in $\sR =
R[x_1,\,\dotsc,\,x_m]$. Thus a monomial in $T_e$ is a monomial of
(total) degree $r(p^e-1)$. 
Note that $T_e=\sR_{r(p^e-1})$ is an $R$-free (left) module with a basis
consisting of monomials of (total) degree $r(p^e-1)$. 
In particular, $T_0 = R$.
\end{itemize}

Fix any $e \in \mathbb N$. We see that $G_{e-1} = G_{e-1}(T)$ is
an $R$-free (left) module with a basis consisting of monomials that
can be expressed as products (under $*$, the multiplication of $T$) of
monomials of degree $r(p^i-1)$ where 
$i \leq e-1$. So all such monomials of total degree $r(p^e-1)$ form an
$R$-basis of $(G_{e-1})_e$. 

In conclusion, $\frac{T_e}{(G_{e-1})_e}$ is free as a left $R$-module
with a basis given by monomials of degree $r(p^e-1)$ that cannot be
written as products (under $*$) of monomials of degree $r(p^i-1)$,
with $i \leq e-1$. We will refer to this basis as the \emph{monomial
basis} of $\frac{T_e}{(G_{e-1})_e}$. By Corollary~\ref{interpretation}, 
we see $c_e(T) = \rank_R(\frac{T_e}{(G_{e-1})_e})$.  

As $c_0(T) = 0$ and $c_1(T) = \rank_R(T_1) = \rank_R(\sR_{r(p-1)})$, we
may assume $e \ge 2$ in the following discussion.

Let $\alpha = (a_1, \dotsc, a_m) \in \mathbb N^m$ such that
$|\alpha| := a_1 + \dotsb + a_m = r(p^e-1)$, so that 
$x^{\alpha} := x_1^{a_1} \cdots x_m^{a_m}$ is a monomial in $T_e$ 
(i.e., of degree $r(p^e-1)$). This monomial
$x^{\alpha}$ belongs to $(G_{e-1})_e$ if and only if it can be decomposed as  
\[
x^{\alpha} = x^{\alpha'} * x^{\alpha''}= x^{\alpha' + p^{e'}\alpha''}
\] 
for some $x^{\alpha'} \in T_{e'},\, x^{\alpha''}\in T_{e''}$ 
with $1 \leq e',\,e'' \le e-1$ and $e' +e'' =e$.
In other words, $x^{\alpha} \in (G_{e-1})_e$ if and only if there is
an equation
\[
{\alpha} = {\alpha' + p^{e'}\alpha''}
\]
for some $\alpha',\,\alpha'' \in \mathbb N^m,\,
1 \leq e' \le e-1,\, e'+e''=e$ with
$|\alpha'| =r(p^{e'}-1)$ and $|\alpha''| = r(p^{e''}-1)$,
which is equivalent to the existence of equations 
\[
a_i = a_i' + p^{e'}a_i'' \quad \text{for all} \quad i \in \{1,\dotsc, m\}
\]
for some $(a_1',\dotsc,a_m'),\,(a_1'',\dotsc,a_m'') \in \mathbb N^m,\, 
1 \leq e' \le e-1,\, e'+e''=e$ with
$\sum_{i=1}^m a'_i =r(p^{e'}-1)$ and $\sum_{i=1}^m a''_i = r(p^{e''}-1)$.
Now it is routine to see that the above holds if and only if there
exist $(a_1',\dotsc,a_m') \in \mathbb N^m$ and $1 \leq e' \le e-1$ with
$\sum_{i=1}^m a'_i =r(p^{e'}-1)$ such that
\[
a_i\trun_{e'} \le a_i' \le a_i \ \text{ and }\  a_i\trun_{e'} = a_i'\trun_{e'}
\quad \text{for all} 
\quad i \in \{1,\dotsc, m\},
\]
which can be seen to be equivalent to the existence of an integer 
$1 \leq e' \leq e-1$ such that 
\[
a_1 \trun_{e'} + \cdots + a_m \trun_{e'} \le r(p^{e'}-1),
\]
which is equivalent to the existence of an integer 
$1 \leq e' \leq e-1$ such that 
\[
\floor{\frac {a_1 \trun_{e'} + \cdots + a_m \trun_{e'}}{p^{e'}}} 
\le \floor{\frac {r(p^{e'}-1)}{p^{e'}}}.
\]
Note that the backward implications of
the last two equivalences rely on the fact that 
$a_1\trun_{e'} + \cdots + a_m \trun_{e'}$ and $r(p^{e'}-1)$
are in the same congruence class modulo $p^{e'}$; the backward
implications of the next to last equivalence also relies on the fact
$a_i\trun_{e'} \equiv a_i \mod p^{e'}$ for all $i$, which allows us to
reverse-engineer $(a_1',\dotsc,a_m') \in \mathbb N^m$ as desired.

With the argument above, we establish the following result. (Again, the fact  
$a_1\trun_{i} + \cdots + a_m \trun_{i} \equiv r(p^{i}-1) \mod p^{i}$
is needed in part~(2) of the following proposition.)

\begin{Prop} \label{prop:good/bad}
Consider $T = T(V_r(R[x_1,\,\dotsc,\,x_m]))$, in prime characteristic $p$.
\begin{enumerate}
\item For any monomial $x_1^{a_1} \cdots x_m^{a_m} \in T_e$ with $e \ge 1$, 
the following are equivalent.
\begin{itemize}
\item $x_1^{a_1} \cdots x_m^{a_m} \in G_{e-1}(T)$.
\item There exists an integer $i$, $1 \le i
\le e-1$, such that the carry-over to the digit associated 
with $p^{i}$ is less than or equal to $\floor{\frac {r(p^{i}-1)}{p^{i}}}$
when $a_1 + \dotsb + a_m$ is calculated in base $p$.
\end{itemize}

\item For any monomial $x_1^{a_1} \cdots x_m^{a_m} \in T_e$ with $e \ge 1$, 
the following are equivalent.
\begin{itemize}
\item $x_1^{a_1} \cdots x_m^{a_m} \notin G_{e-1}(T)$.
\item $a_1 \trun_{i} + \cdots + a_m \trun_{i} = r(p^{i}-1) +
  d_{i}p^{i}$ with $1 \le d_{i} \in \mathbb N$ for all $1 \le i \le e-1$.
\item The carry-over to the digit associated with $p^{i}$ is greater than
$\floor{\frac {r(p^{i}-1)}{p^{i}}}$ for all $1 \le i \le e-1$ 
when $a_1 + \dotsb + a_m$ is calculated in base $p$. 
\end{itemize}
\end{enumerate} 
\end{Prop}

\begin{Prop} \label{prop:count}
For $T = T(V_r(R[x_1,\,\dotsc,\,x_m]))$, $c_e(T)$ is the number of monomials 
$x_1^{a_1} \cdots x_m^{a_m} \in T_e$ such that
the carry-over to the digit associated 
with $p^{i}$ is bigger than $\floor{\frac {r(p^{i}-1)}{p^{i}}}$
for all $1 \le i\le e-1$ 
when $a_1 + \dotsb + a_m$ is calculated in base $p$. 
\end{Prop}

Using the criteria given in Proposition~\ref{prop:good/bad}, we are
able to determine precisely when  
$T(V_r(R[x_1,\,\dotsc,\,x_m]))$ is finitely generated over $T_0 = R$.

\begin{Thm}\label{thm:mr}
Let $T = T(V_r(R[x_1,\,\dotsc,\,x_m]))$, with $r,\,m,\,R$ as above. 
\begin{enumerate}
\item If $r \ge m-1$, then $T$ is generated by $T_1$ over $T_0$ (that
  is, $c_e(T) = 0$ for all $e \ge 2$). 
\item If $r < m-1$, then $c_e(T) > 0$ (i.e.,
  $T_e$ is not generated by lower degree) for all $e \ge 1$.
\item The ring $T(V_r(R[x_1,\,\dotsc,\,x_m]))$ is finitely generated over
$R$ if and only if $r \ge m-1$.
\end{enumerate}
\end{Thm}

\begin{proof}
Evidently, we only need to prove (1) and (2).

(1) Suppose, on the contrary, that for some $e \ge 2$ there exists a
monomial $x_1^{a_1} \cdots x_m^{a_m} \in T_e$ that does not belong to
$G_{e-1}(T)$. Then by Proposition~\ref{prop:good/bad}
\[
a_1 \vert_{i} + \cdots + a_m\vert_{i} \ge r(p^{i}-1) + p^{i}
\]
for all $1 \le i \le e-1$. However, the assumption $r \ge m-1$
implies
\[
a_1 \trun_{i} + \cdots + a_m \trun_{i} \le m(p^{i}-1) \le
(r+1)(p^{i}-1) < r(p^{i}-1) + p^{i}.
\]
We get a contradiction.

(2) As $c_1(T) > 0$ is clear, we assume $e \ge 2$. Consider
\[
x_1^{p^e-1}\dotsm x_{r-1}^{p^e-1}
x_r^{p^e-p^{e-1}-1}x_{r+1}^{p^{e-1}-1}x_{r+2}^{1} 
\in \sR_{r(p^e-1)} = T_e.
\]
Now it is routine to see that the carry-over to the digit associated
with $p^{i}$ is $\floor{\frac {r(p^{i}-1)}{p^{i}}} +1$
for all $1 \le i \le e-1$ when $a_1 = p^e-1,\, \dotsc,\,a_{r-1} =
p^e-1$, $a_r = p^e-p^{e-1}-1$, $a_{r+1}={p^{e-1}-1}$, $a_{r+2}={1}$
and $a_i = 0$ (for $r+2 < i \le m$) are added up in base $p$. This
verifies 
$x_1^{p^e-1}\dotsm x_{r-1}^{p^e-1}x_r^{p^e-p^{e-1}-1}x_{r+1}^{p^{e-1}-1}x_{r+2} 
\notin G_{e-1}(T)$ and hence $c_e(T) > 0$. 
\end{proof}

\section{Computing $c_e(T(V_r(R[x_1,\,\dotsc,\,x_m])))$}
\label{sec:computing-ce}

Let $R$, $m$, $r$, $\sR$, $\sV$ and $T$ be as in last section and keep
the notations. 
In particular, $T = T(V_r(R[x_1,\,\dotsc,\,x_m]))$ is an $\mathbb
N$-graded ring. For simplicity, denote $c_e(T)$ by $c_{m,r,e}$ or
simply by $c_e$ since $r$ and $m$ are understood. (It should be clear
that $c_e(T(V_r(R[x_1,\,\dotsc,\,x_m])))$ is independent of $R$. Also
note that $c_1 = \rank_R (\sR_{r(p-1)}) = \binom{r(p-1)+m-1}{m-1}$.) 


Fix an integer $e \ge 2$. 
The goal is to count the number of monomials that produce
the monomial basis of $\frac{T_e}{(G_{e-1})_e}$. 

First, we set up some notations. 
Let $\alpha = (a_1, \dotsc, a_m) \in \mathbb N^m$ with 
$|\alpha| := a_1 + \dotsb + a_m = r(p^e -1)$. 
For each $n \in [1,\,m] := \{1,\dotsc,m\}$, write 
$a_n = \ol{\cdots a_{n,i} \cdots a_{n,0}}$ 
in base $p$ expression. Then, for each $i \in [0,\,e-2] := \{0, \dotsc,
e-2\}$, denote  
\[
\alpha_i := (a_{1,i}, \dotsc, a_{m, i})\in \mathbb N^m, 
\]
which can be referred to as the vector of the digits corresponding to
$p^i$. Also denote
\[
\alpha_{e-1} := 
\left(\floor{\frac{a_{1}}{p^{e-1}}}, \dotsc,\floor{\frac{a_{m}}{p^{e-1}}}\right)
= \left(a_1 - a_1 \trun_{e-1}, \dotsc, a_m - a_m \trun_{e-1}\right)
\in \mathbb N^m.
\]
Moreover, for each $i \in \{0,\dotsc,e-1\}$, let $f_i(\alpha)$ denote the
carry-over to the digit corresponding to $p^{i}$ when computing
$\sum_{i=1}^ma_i$ in base $p$. In other words, 
\[
f_i(\alpha) := \floor{\frac{a_1 \trun_{i} + \dotsb + a_m \trun_{i}}{p^i}}. 
\]
Note that $f_{0}(\alpha) = 0$. 
Then denote $f(\alpha) :=(f_{e-1}(\alpha), \dotsc, f_{0}(\alpha)) 
\in \mathbb N^e$. 
Finally, denote 
\[
d(\alpha) := (d_{e-1}(\alpha), \dotsc, d_0(\alpha)) : =
f(\alpha) - \left(\floor{\frac {r(p^{e-1}-1)}{p^{e-1}}},
\dotsc, \floor{\frac {r(p^{0}-1)}{p^{0}}}\right)\in \mathbb Z^e,
\]
so that $d_i(\alpha) = f_i(\alpha)-\floor{\frac {r(p^{i}-1)}{p^{i}}}$ for
all $i \in [0,\,e-1] := \{0,\dotsc,e-1\}$. Note that $d_0(\alpha) = 0$. 
Moreover, for all $i \in [0,\,e-2]$, we have 
\begin{align*}
d_{i+1}(\alpha) 
& = 
\floor{\frac{a_1 \trun_{i+1} + \dotsb + a_m \trun_{i+1}}{p^{i+1}}} - 
\floor{\frac {r(p^{i+1}-1)}{p^{i+1}}}  \\
& \overset{\dagger}{=} 
\floor{\frac{|\alpha_i| + f_i(\alpha)}{p}} -
\floor{\frac{r(p-1) + \floor{\frac {r(p^{i}-1)}{p^{i}}}}{p}}  \\
& \overset{\ddagger}{=} 
\frac 1p \left[(|\alpha_i| + f_i(\alpha)) -
\left(r(p-1) + \floor{\frac {r(p^{i}-1)}{p^{i}}}\right)\right] \\
& = 
\frac 1p \left[|\alpha_i| 
+ \left(f_i(\alpha))- \floor{\frac {r(p^{i}-1)}{p^{i}}}\right) - r(p-1)\right]\\
& = 
\frac 1p \big[|\alpha_i| + d_i(\alpha) - r(p-1)\big]. 
\end{align*}
Note that $\overset{\dagger}{=}$ follows from how we compute the
carry overs to digit corresponding $p^{i+1}$, while
$\overset{\ddagger}{=}$ follows from the fact that 
$|\alpha_i| + f_i(\alpha) \equiv r(p-1) + \floor{\frac
  {r(p^{i}-1)}{p^{i}}} \mod p$ since they are all congruent to the
(same) number representing the digit associated with $p^i$ in the base
$p$ expression of $r(p^e-1)$ and $r(p^i-1)$. 

Let $\alpha = (a_1, \dotsc, a_m) \in \mathbb N^m$ with 
$|\alpha| = r(p^e -1)$ as above
and let $\delta = (d_{e-1}, \dotsc,d_0)\in \mathbb Z^e$ with $d_{0} = 0$. 
By what we have established above, we see  
\begin{align*}
d(\alpha) = d 
& \iff d_i(\alpha) = d_i,\,\forall i \in [1,\,e-1] \\
& \iff d_{i+1}(\alpha) = d_{i+1},\,\forall i \in [0,\,e-2] \\
& \iff \frac 1p \big[|\alpha_i| + d_i(\alpha) - r(p-1)\big] 
   = d_{i+1},\,\forall i \in [0,\,e-2]  \\
& \iff 
|\alpha_i| + d_i(\alpha) - r(p-1) = d_{i+1}p,\,\forall i \in [0,\,e-2] \\
& \iff 
|\alpha_i| + d_i(\alpha) = r(p-1) + d_{i+1}p,\,\forall i \in [0,\,e-2] \\
& \overset{*}{\iff} 
|\alpha_i| + d_i = r(p-1) + d_{i+1}p,\,\forall i \in [0,\,e-2]  \\
& \iff 
|\alpha_i| = r(p-1) + d_{i+1}p - d_i,\,\forall i \in [0,\,e-2].
\end{align*}
Note that $\overset{*}{\implies}$ holds because the assumption (i.e.,
antecedent) of this implication already implies $d(\alpha) =\delta$,  
while $\overset{*}{\impliedby}$ follows from an easy induction on
$i$ (in light of the established equation
$d_{i+1}(\alpha)=\frac 1p \big[|\alpha_i| + d_i(\alpha) - r(p-1)\big]$). 
Furthermore, the assumption $|\alpha| = r(p^e -1)$ (together with
$d(\alpha) = \delta$) translates to the following
\[
|\alpha_{e-1}| + f_{e-1}(\alpha) 
= \floor{\frac{a_1 + \dotsb + a_m}{p^{e-1}}}
= \floor{\frac{r(p^e-1)}{p^{e-1}}}
= r(p-1) + \floor{\frac {r(p^{e-1}-1)}{p^{e-1}}},
\] 
which is obtained by examining summations $a_1 + \dotsb + a_m$ and 
$\overbrace{(p^e-1) + \dotsb + (p^e-1)}^{r \text{ terms}}$ in base $p$. 
Therefore 
\[
|\alpha_{e-1}| 
= r(p-1) + \floor{\frac {r(p^{e-1}-1)}{p^{e-1}}} - f_{e-1}(\alpha) 
= r(p-1)-d_{e-1}(\alpha) = r(p-1)-d_{e-1}.
\]

In summary, with $\alpha \in \mathbb N^m$ and $\delta \in \mathbb Z^e$ with
$d_0 = 0$ as above, we conclude that
$|\alpha| = r(p^e-1)$ and $d(\alpha) = \delta$ if and only if  
\[
|\alpha_{e-1}| = r(p-1)-d_{e-1}
\quad \text{and} \quad 
|\alpha_i| = r(p-1) + d_{i+1}p - d_{i}  \quad 
\text{for all} \quad i \in \{0, \dotsc, e-2\}.
\]

Now we are ready to formulate $c_e = c_e(T)$ for 
$T = T(V_r(R[x_1,\,\dotsc,\,x_m]))$. This result generalizes
\cite[Proposition~3.7]{EY}. Since $c_e = 0$ for all $e \le 2$ when $m
\le r+1$, the formula in the following proposition is most meaningful
when $m-r-1 > 0$. 

\begin{Prop}\label{prop:c_e}
For $T = T(V_r(R[x_1,\,\dotsc,\,x_m]))$, we have the following formula:
\begin{align*}
c_{e}&= \sum_{\substack{
(d_{e-1},\, \dotsc,\, d_1,\, d_{0}=0) \in \mathbb N^{e}
\\d_i \ge 1  \text{ for } 1 \le i \le e-1}}
\left(P_m\left(r(p-1)-d_{e-1}\right) 
\prod_{i=0}^{e-2} M_{p,m}(r(p-1) + d_{i+1}p - d_{i})\right) \\
&= \sum_{\substack{
(d_{e-1},\, \dotsc,\, d_1,\, d_{0}=0) \in \mathbb N^{e}
\\1 \le d_i \le m-r-1 \text{ for } 1 \le i \le e-1}}
\left(\binom{r(p-1)-d_{e-1}+m-1}{m-1}
\prod_{i=0}^{e-2} M_{p,m}(r(p-1) + d_{i+1}p - d_{i})\right)
\end{align*}
for all $e \ge 2$, 
where $P_m(i)$ denotes $\rank_R(R[x_1,\,\dotsc,\,x_m]_i)$, i.e., 
$P_m(i) = \binom{m+i-1}{i} = \binom{m+i-1}{m-1}$. 
\end{Prop}

\begin{proof}
Fix any $e \ge 2$ and adopt the notations set up above. 
Consider $x^{\alpha} = x_1^{a_1} \dotsm x_m^{a_m} \in T_e$. By
Proposition~\ref{prop:good/bad}, $x^{\alpha} \notin G_{e-1}(T)$ if and 
only if
\[
d_i(\alpha) \ge 1 \quad \text{for all} \quad i \in \{1,\,\dotsc,\,e-1\}.
\]
To determine $c_e$, we need to find the number of monomials with the
above property, as stated in Proposition~\ref{prop:count}.
This is equivalent to counting the number of $\alpha \in \mathbb N^m$
such that $|\alpha| = r(p^e -1)$ and $d_i(\alpha) \ge 1$ for all 
$i \in [1,\,e-1]$.

Fix any $\delta = (d_{e-1}, \dotsc,d_0)\in \mathbb N^e$ with $d_{0} = 0$
and $d_i \ge 1$ for all $i \in [1,\,e-1]$. We intend to find the
number of $\alpha \in \mathbb N^m$ such that $|\alpha| = r(p^e -1)$
and $d(\alpha) = \delta$, which can be written as
\[
\card{\{\alpha \in \mathbb N^m : 
|\alpha| = r(p^e -1) \text{ and } d(\alpha) = \delta\}},
\] 
in which $\card X$ stands for the cardinality of any set $X$.  

For each $i \in [1,\,e-2]$, the number of ways to realize 
$|\alpha_i| = r(p-1) + d_{i+1}p - d_{i}$ is given as follows:
\[
\card{\{\alpha_i \in [0,\,p-1]^m : 
|\alpha_i| = r(p-1) + d_{i+1}p - d_{i}\}}
=M_{p,m}(r(p-1) + d_{i+1}p - d_{i}).
\] 
The number of ways to realize 
$|\alpha_{e-1}| = r(p-1) - d_{e-1}$ is given as follows:
\[
\card{\{\alpha_{e-1} \in N^m : 
|\alpha_{e-1}| = r(p-1) - d_{e-1}\}}
=P_{m}(r(p-1) - d_{e-1}).
\] 

Therefore, the number of $\alpha \in \mathbb N^m$ such that 
$|\alpha| = r(p^e -1)$ and $d(\alpha) = \delta$ is governed by the
following formula:  
\begin{multline*}
\card{\{\alpha \in \mathbb N^m : 
|\alpha| = r(p^e -1) \text{ and } d(\alpha) = \delta\}} \\
=P_m\left(r(p-1)-d_{e-1}\right)\prod_{i=0}^{e-2} M_{p,m}(r(p-1) + d_{i+1}p - d_{i}).
\end{multline*}
Observe that if $m-r-1 \le 0$, then 
\[
\card{\{\alpha \in \mathbb N^m : 
|\alpha| = r(p^e -1) \text{ and } d(\alpha) = \delta\}} = 0,
\]
which follows from $M_{p,m}(r(p-1) + d_{1}p - d_{0}) = 0$ since
$r(p-1) + d_{1}p - d_{0} \ge r(p-1) + p = (r+1)(p-1)+1 \ge 
m(p-1)+1$; also see Theorem~\ref{thm:mr}(1).  
We further observe that, whenever there exists $d_i > m-r-1 >0$ for
some $i \in [1,\,e-1]$, then 
\[
\card{\{\alpha \in \mathbb N^m : 
|\alpha| = r(p^e -1) \text{ and } d(\alpha) = \delta\}} = 0.
\]
Indeed, pick the least $i \in [1,\,e-1]$ such that $d_i > m-r-1 >0$
and we get $r(p-1) + d_{i}p - d_{i-1} \ge m(p-1)+1$ and hence 
$M_{p,m}(r(p-1) + d_{i}p - d_{i-1}) = 0$.
Put differently, when adding $m$ many non-negative integers to
$r(p^e-1)$, the carry overs to digits associated with $p^i$ can not exceed 
$\floor{\frac {r(p^{i}-1)}{p^{i}}} + m - r-1$. 

Finally, exhausting all $\delta = (d_{e-1}, \dotsc,d_0)\in \mathbb N^e$ 
with $d_{0} = 0$ and $d_i \ge 1$ for $i \in [1,\,e-1]$, we can
formulate $c_{e}=c_e(T(R[x_1, \dotsc, x_m]))$ as follows:
\begin{align*}
c_{d,e} & 
= \sum_{\substack{
(d_{e-1},\, \dotsc,\, d_1,\, d_{0}=0) \in \mathbb N^{e}
\\d_i \ge 1  \text{ for } 1 \le i \le e-1}}
\card{\{\alpha \in \mathbb N^m : |\alpha| = r(p^e-1) \text{ and } 
d(\alpha) = (d_{e-1},\, \dotsc,\, d_1,\, d_{0})\}} \\
& = \sum_{\substack{
(d_{e-1},\, \dotsc,\, d_1,\, d_{0}=0) \in \mathbb N^{e}
\\d_i \ge 1  \text{ for } 1 \le i \le e-1}}
\left(P_m\left(r(p-1)-d_{e-1}\right)
\prod_{i=0}^{e-2} M_{p,m}(r(p-1) + d_{i+1}p - d_{i})\right)\\
&= \sum_{\substack{
(d_{e-1},\, \dotsc,\, d_1,\, d_{0}=0) \in \mathbb N^{e}
\\1 \le d_i \le m-r-1 \text{ for } 1 \le i \le e-1}}
\left(\binom{r(p-1)-d_{e-1}+m-1}{m-1}
\prod_{i=0}^{e-2} M_{p,m}(r(p-1) + d_{i+1}p - d_{i})\right),
\end{align*}
which verifies the equations.
\end{proof}

Next, we outline a method that allows us compute $c_{e} =
c_e(T(V_r(R[x_1,\,\dotsc,\,x_m])))$ for any $m,\,r$ with $m \ge r+2$,
in which $R$ may have any prime characteristic $p$. 
(Note that, if $m \le r+1$, then $c_{e}= 0$ for all $e \ge 2$, see
Theorem~\ref{thm:mr}.) 
The following generalizes \cite[Discussion~3.8]{EY}.

\begin{Dis}\label{dis:m,r,p}
Fix any positive integers $r,\,m$ such that $r+1 < m$, any prime
number $p$, and any ring $R$ with characteristic $p$. 
Let $\sR=R[x_1,\,\dotsc,\,x_m]$.
We describe a way to determine $c_{e} = c_e(T(V_r(\sR)))$ explicitly
as follows: 

For every $e \ge 0$, denote
\[
X_e :=
\begin{bmatrix}
X_{e,1} \\
\vdots \\
X_{e,m-r-1}
\end{bmatrix}_{(m-r-1) \times 1},
\]
in which
\begin{align*}
X_{e,n} &:= \sum_{\substack{
(d_{e+1}=n,\, d_{e-1},\,\dotsc,\, d_1,\,d_{0}=0) \in \mathbb N^{e+2}
\\1 \le d_i \le m-r-1 \text{ for } 0 \le i \le e}}
\prod_{i=0}^{e} M_{p,m}(r(p-1) + d_{i+1}p - d_{i})
\end{align*}
for all $n \in \{1,\,\dotsc,\,m-r-1\}$.

With these notations, it is straightforward to see that, for all 
$i \in [1,\,m-r-1]$,  
\begin{align*}
X_{e+1,i} &= \sum_{j=1}^{m-r-1}M_{p,m}(r(p-1) + ip - j)X_{e,j}
\end{align*}
In other words, $X_{e+1}$ can be computed recursively: 
\[
X_{e+1} = U \cdot X_e,
\]
where
\[
U:= \begin{bmatrix}
u_{ij}
\end{bmatrix}_{(m-r-1) \times (m-r-1)} 
\quad \text{with} \quad 
u_{ij}:=M_{p,m}(r(p-1) + ip - j).
\]

Therefore, 
\[
X_e = U^{e}\cdot  X_0 \quad \text{for all} \quad e \ge 0.
\]

With $m,\,r$ and $p$ given, both $X_0$ and 
$U= (u_{ij})_{(m-r-1) \times (m-r-1)}$ can be determined explicitly. 
Accordingly, we can compute $X_e = U^{e}\cdot X_0$ explicitly for
all $e \ge 0$. 

Finally, for all $e \ge 2$, we can determine  $c_{e} =c_{e}(T(V_r(\sR)))$
explicitly, as follows: 
\begin{align*}
c_{e} 
&= \sum_{\substack{
(d_{e-1},\, \dotsc,\, d_1,\, d_{0}=0) \in \mathbb N^{e}
\\1 \le d_i \le m-r-1 \text{ for } 1 \le i \le e-1}}
\left(P_{m}(r(p-1)-d_{e-1})\prod_{i=0}^{e-2} M_{p,m}(r(p-1) 
+ d_{i+1}p - d_{i})\right)\\
&= \sum_{n = 1}^{m-r-1}\left(P_{m}(r(p-1)-n)
\sum_{\substack{
(d_{e-1}=n,\, \dotsc,\, d_1,\, d_{0}=0) \in \mathbb N^{e}
\\1 \le d_i \le m-r-1 \text{ for } 1 \le i \le e-2}}
\prod_{i=0}^{e-2} M_{p,m}(r(p-1) + d_{i+1}p - d_{i})\right)\\
&=\sum_{n=1}^{m-r-1} P_{m}(r(p-1)-n) X_{e-2,n} 
= \sum_{n=1}^{m-r-1} \binom{r(p-1)-n+m-1}{m-1} X_{e-2,n}\\
&=Y_0 \cdot U^{e-2} \cdot X_0,
\end{align*}
where $Y_0 := \begin{bmatrix}
\binom{r(p-1)-1+m-1}{m-1} & \cdots & \binom{r(p-1)-(m-r-1)+m-1}{m-1} \\
\end{bmatrix}_{1 \times (m-r-1)}$.
Consequently, $\cx(T(V_r(\sR)))$ can be computed.
\end{Dis}

\begin{Def}
In what follows, we call 
\[
U(p,r,m)=U:= \begin{bmatrix}
u_{ij}
\end{bmatrix}_{(m-r-1) \times (m-r-1)}
\quad \text{with} \quad 
u_{ij}:=M_{p,m}(r(p-1) + ip - j)
\]
as the {\it determining matrix} for $p, r, m$.
\end{Def}

\begin{Thm}\label{thm:m=r+2}
Consider $T = T(V_r(R[x_1,\dotsc,x_m]))$ as above with $m=r+2$. Then 
$c_e(T) = \binom{rp}{m-1}\binom{p+m-2}{m-1}^{e-2}\binom{p+m-3}{m-1}$
for all $e \ge 2$ and $\cx(T) = \binom{p+m-2}{m-1}$. 
\end{Thm}

\begin{proof}
Adopting all the notations introduced in Discussion~\ref{dis:m,r,p},
we see
\begin{align*}
&X_0 = M_{p,m}(r(p-1) + p) = M_{p,m}(p-2) = P_m(p-2) = \binom{p+m-3}{m-1} >0, \\
&U = M_{p,m}((r+1)(p-1)) = M_{p,m}(p-1) = P_m(p-1) =\binom{p+m-2}{m-1} >0,\\
&Y_0 = P_m(r(p-1)-1) = \binom{r(p-1)-1+m-1}{m-1} = \binom{rp}{m-1}> 0. 
\end{align*}
Here we use the fact $M_{p,m}(i) = M_{p,m}(m(p-1)-i)$ for all $i$.
Therefore, for all $e \ge 2$, we obtain
\[
c_e =\binom{r(p-1)+m-2}{m-1}\binom{p+m-2}{m-1}^{e-2}
\binom{p+m-3}{m-1}, 
\]
which establishes
\[
\cx(T(V_r(R[x_1,\dotsc,x_m))) = \binom{p+m-2}{m-1}
\]
when $m=r+2$.
\end{proof}

\subsection{The Frobenius complexity as $p \to \infty$.} \label{p-to-infinity}

We will maintain the notations from this section, including the
condition $m \geq r+2$ and $r > 0$. The following results are
straightforward and left to the reader. We will comment on their
proofs only when necessary. 

\begin{Lma}
\label{calc}
Fix $m>0$ an integer and $p$ a prime number.

\begin{enumerate}
\item
$M_{p,m}(i)= M_{p.m}(m(p-1)-i).$
\item
$M_{p,m}(i) \leq M_{p,m}(j)$ if $0 \leq i \leq j \leq \lceil m(p-1)/2
\rceil$ or  $\lceil m(p-1)/2 \rceil \leq j \leq i \leq m(p-1)$. 
\end{enumerate}
\end{Lma}

\begin{Lma}
\label{calc2}
For any integers $i, j$ such that $1 \leq i,j \leq m-r-1$, we have
$$ p -3 < p \leq r(p-1)+pi-j \leq m(p-1)-p+3,$$ for all $p \gg 0$.
\end{Lma}

\begin{Def}
For any $t \times s$ matrix $A=(a_{ij})$ with nonnegative entries,
where $t,s$ are positive integers, define $\abs A = \min \{ a_{ij}
\}$ and $\Vert A  \Vert = \max \{ a_{ij} \}$. 
\end{Def}

The following Lemma is a consequence of Lemmata~\ref{calc} and~\ref{calc2}.

\begin{Lma}
\label{abs}
Given $m$ and $r$, we have the following inequalities:
$$\binom{m-1+p-3}{m-1} \leq \abs U \leq \Vert U \Vert \leq 
\binom{m-1+ \lceil \frac{m(p-1)}{2} \rceil}{m-1}$$
for the determining matrix $U = U(p,r,m)$ for all $p \gg 0$.
\end{Lma}

\begin{Lma}
Let $A, B$ be matrices with nonnegative entries of sizes $l \times t$,
respectively $t \times s$, with $l, t, s$ positive integers. Then 
\[
t \abs A \cdot \abs B \le \abs {A \cdot B}  \le  
\Vert{A \cdot B} \Vert \leq t \Vert A \Vert \cdot \Vert B \Vert.
\]
\end{Lma}

Now, let us recall that (cf.~Discussion~\ref{dis:m,r,p})
$$c_e = Y_0 \cdot X_{e-2}= Y_0 \cdot U^{e-2} \cdot X_0,$$ 
where
\[
X_0 =
\begin{bmatrix}
X_{0,1} \\
\vdots \\
X_{0,m-r-1}
\end{bmatrix}_{(m-r-1) \times 1}
\quad \text{with} \quad 
X_{0,i} = M_{p,m}( r(p-1)+ip)
\]
and  
\[
Y_0 = \begin{bmatrix}
\binom{r(p-1)-1+m-1}{m-1} & \cdots & \binom{r(p-1)-(m-r-1)+m-1}{m-1} \\
\end{bmatrix}_{1 \times (m-r-1)}.
\]


\begin{Lma}
\label{pos}
For all $p,\,m,\,r$ as above, both $X_0$ and $Y_0$ are non-zero.
\end{Lma}

\begin{proof}
Indeed, $m \ge r+2$ implies $0 \le r(p-1)+p < r(p-1)+2(p-1) \le
m(p-1)$, which implies $X_{0,1} = M_{p,m}(r(p-1)+p) > 0$.

On the other hand, $r(p-1)-1 \ge 0$ implies $r(p-1)-1+m-1 \ge m-1$,
which implies $\binom{r(p-1)-1+m-1}{m-1} > 0$.
\end{proof}

Moreover, both $X_0$ and $Y_0$ have all positive entries
for $p \gg 0$. In fact we can be more precise. 

\begin{Lma}
\label{pos}
If $p \geq m-r$, then both $X_0$ and $Y_0$ have all positive entries.
\end{Lma}

\begin{proof}
If $p \geq m-r$, then $0 \leq r(p-1)+ip \leq  m(p-1)$ and hence  
$M_{p,m}(r(p-1)+ip) > 0$, for all $i =1, \dotsc, m-r-1$. 

On the other hand, note that $r(m-r)- m+1  = 
-(r-1)(r-m+1) \ge 0$ for all $r = 1, \dotsc, m-2$.  
Consequently, if $p \ge m-r$ then for all $i = 1, \dotsc, m-r-1$,
\begin{align*}
r(p-1)-i+m-1 &\ge r(p-1)-(m-r-1) +m-1 \\
&=(rp-m+1) + m-1 \ge (r(m-r) - m +1) + m-1 \ge m-1, 
\end{align*}
which leads to $\abs{Y_0} > 0$. 
\end{proof}

\begin{Prop} \label{prop:bounds}
We have
$$c_e \leq (m-r-1)^{e-1}  \cdot \Vert Y_0 \Vert 
\cdot \Vert U \Vert ^{e-2} \cdot \Vert X_0 \Vert$$
and
$$(m-r-1)^{e-1}  \cdot \abs{Y_0}  \cdot \abs{U}^{e-2} \cdot \abs{X_0}
\leq c_e.$$
(In fact $(m-r-1)^{e-3}  \cdot \Abs{Y_0}  \cdot \abs{U}^{e-2} \cdot \Abs{X_0}
\leq c_e$.)
Therefore we have that
$$ (m-r-1) \abs{U} \leq \cx (T(V_r(\sR)) \leq (m-r-1)\Vert U \Vert$$
for $p \gg 0$, where $\sR = R[x_1,\,\dotsc,\,x_m]$. 
\end{Prop}

\begin{Cor} \label{cor:limit}
Let $\sR = R[x_1,\,\dotsc,\,x_m]$. If $p \gg 0$, then
$$ (m-r-1)\binom{m-1+p-3}{m-1}  \leq \cx (T(V_r(\sR)) \leq 
(m-r-1) \binom{m-1+ \lceil \frac{m(p-1)}{2} \rceil}{m-1}$$ 
and therefore
$\lim_{p \to \infty} \log_p\cx (T(V_r(\sR)) =  m-1$.
\end{Cor}

This corollary motivates the definition of Frobenius complexity in characteristic zero, which is given in Section 4, see Definition~\ref{FC0}.


\subsection{Perron-Frobenius} \label{subsec:pf}
We would like to summarize a few things
about square matrices with positive real entries. Any such matrix
admits a real positive eigenvalue $\lambda$ such that all other
eigenvalues have absolute value less than $ \lambda$. We will refer to
this eigenvalue as the Perron root or Perron-Frobenius eigenvalue of
the matrix. This eigenvalue is a simple root of the characteristic
polynomial of the matrix. Moreover, an eigenvector for
$\lambda$ either has all entries positive or has all entries
negative. See \cite{Pe} and \cite{Fr}.

Let $p \gg 0$. Since $U$ has only positive entries by
Lemma~\ref{abs}, let $\lambda$ be the Perron-Frobenius eigenvalue for $U$.
There exists an invertible matrix $P$ such that 
$$U = P D P^{-1}$$ 
where $D$ is the Jordan canonical form of $U$. (We may also take $D$
to be the rational canonical form of $U$ over $\mathbb R$ if we prefer
to stay within $\mathbb R$.) 
Without loss of generality, the left upper
corner of $D$ is $\lambda$ (thus all the other entries of the first
row or first column are $0$); that is,
\[
D = \begin{bmatrix}
\lambda & 0 \\
0 & D_1
\end{bmatrix}_{(m-r-1) \times (m-r-1)}
\]
with $D_1$ being an $(m-r-2) \times (m-r-2)$ matrix whose eigenvalues
are all less than $\lambda$ in absolute value. 
Hence the first column (row) of $P$
($P^{-1}$) is an eigenvector of $U$ ($U^T$) for $\lambda$.  
Thus, without loss of generality, we may assume that the first column of
$P$ and (consequently) the first row of $P^{-1}$ have all positive entries.  


Lastly, since both $Y_0$ and $X_0$ are non-zero, the first entries of
both $Y_0 P$ and $P^{-1} X_0$ are positive. 
Write $Y_0 P = [a, \ A]$ and $P^{-1} X_0 = [b, \ B]^T$ in block form.
Now the fact that $\lambda$ is the largest eigenvalue in absolute
value implies that 
\[
c_e = Y_0 U^{e-2} X_0 = (Y_0 P) D^{e-2} (P^{-1} X_0) 
= ab\lambda^{e-2} + AD_1^{e-2}B
=ab\lambda^{e-2}+o(\lambda^{e}). 
\] 
Thus 
\[
\cx(T(V_r(\sR)))= \lambda. 
\]
(The above argument applies as long as $p,\,m,\,r$ are such that
$U$ is all positive, since $X_0$ and $Y_0$ are always non-zero.)  

\section{Frobenius complexity of determinantal rings} 

In this section, we combine what we have obtained to derive results on
the Frobenius complexity of determinantal rings. 
In particular, we translate the results on $T(V_r(R[x_1,\dotsc,x_m]))$ to
$S_{m,n}$ with $m > n \ge 2$.

\begin{Thm}\label{thm:f-com-det}
Let $K$, $S_{m,n}$ and $\sR_{m,n}$ be as in Section~\ref{sec:det} 
(cf.~Definition~\ref{def:segre-complete}) with $m > n \ge 2$. 
Further assume that $K$ is a field of prime characteristic $p$. 
Let $E_{m,n}$ denote the injective hull of the residue field of $S_{m,n}$. 

\begin{enumerate}
\item The ring of Frobenius operators of $S_{m,n}$ (i.e.,
  $\sF(E_{m,n})$) is never finitely generated over $\sF_0(E_{m,n})$.
\item When $n = 2$, we have $\cx_F(S_{m,2}) = \log_p\binom{p+m-2}{m-1}$.
\item We have $\lim_{p \to   \infty}\cx_F(S_{m,n}) = m-1$. 
\item For $p \gg0$ or whenever the determining matrix $U=U(p, m, m-n)$
has all positive entries, we have $\cx_F(S_{m,n}) = \log_p (\lambda)$, 
in which $\lambda$ is the Perron root for $U$.  
\end{enumerate}
\end{Thm}

\begin{proof}
(1) Since $m-n \le m-2$, we see that $T(V_{m-n}(K[x_1,\dotsc,x_{m}]))$ 
is not finitely generated over $T_0(V_{m-n}(K[x_1,\dotsc,x_{m}]))$ by 
Theorem~\ref{thm:mr}(2). 
Thus $\sF(E_{m,n})$) is not finitely generated over $\sF_0(E_{m,n})$
by Theorem~\ref{thm:segre}(1).

(2) By Theorem~\ref{thm:segre}(3) and Theorem~\ref{thm:m=r+2}, 
\[
\cx_F(S_{m,2}) = \log_p\cx(T(V_{m-2}(K[x_1,\dotsc, x_{m}])))
= \log_p\binom{p+m-2}{m-1}.
\qedhere
\]

(3) This follows from Corollary~\ref{cor:limit}.

(4) This is a straightforward consequence of the discution in
Subsection~\ref{subsec:pf}.
\end{proof}

\begin{Rem} We like to point out the following:
\begin{enumerate}
\item Also note that, for  every $m > 2$,
\[
\lim_{e \to \infty} c_e(\sF(E_{m,2})) 
= \lim_{e \to \infty} c_e(T(V_{m-2}(K[x_1, \dotsc,x_m]))) = \infty.
\]
\item Moreover, there exists an onto (hence nearly onto) graded ring
  homomorphism from 
$T(V_{m-n}(K[x_1, \dotsc,x_m]))$ to $T(V_{m-n}(K[x_1,\dotsc,x_{m-n+2}]))$. 
Thus by Corollary~\ref{cor:nearly-onto}, 
\[
c_e(T(V_{m-n}(K[x_1, \dotsc,x_m]))) 
\ge c_e(T(V_{m-n}(K[x_1,\dotsc,x_{m-n+2}])))
\] 
for all $ e \ge 0$. 
Hence $c_e(\sF(E_{m,n})) \ge c_e(\sF(E_{m-n+2,2}))$ for all $ e \ge 0$
and consequently
\[
\lim_{e \to \infty} c_e(\sF(E_{m,n})) = \infty
\] 
for all $m > n \ge 2$. 
\end{enumerate}
\end{Rem}

\subsection{Example}

We will illustrate our method with a concrete example. We are going to
use freely the notations established so far (especially the ones in
Section~\ref{sec:computing-ce}).  
 
Let $r=2,\, m=5$ and $K$ be a field of characteristic $p=3$. 
We are going to compute $c_e = c_e(T(V_2(K[x_1, \ldots, x_5])))$,
which in turn equals $c_e(\sF(E_{5,3}))$ by Theorem~\ref{thm:segre}.
As in Discussion~\ref{dis:m,r,p}, we have
\[
X_e= U^e\cdot X_0 \quad \text{for all} \quad e \geq 0,
\]
in which
\begin{align*}
X_e &= \begin{bmatrix}
X_{e,1} \\
X_{e,2}
\end{bmatrix}, \\
X_0 &= \begin{bmatrix}
X_{0,1} \\
X_{0,2}
\end{bmatrix}
=\begin{bmatrix}
M_{3,5}(7) \\
M_{3,5}(10)
\end{bmatrix}
=\begin{bmatrix}
30 \\
1
\end{bmatrix}, \\
U &= \begin{bmatrix}
M_{3,5}(6) & M_{3,5}(5) \\
M_{3,5}(9) & M_{3,5}(8)
\end{bmatrix}
= \begin{bmatrix}
45 & 51  \\
5  & 15
\end{bmatrix}.
\end{align*}

Note that $U$ has all positive entries and the eigenvalues of $U$ are
$2(15+ 2\sqrt{30})$ and $2(15-2\sqrt{30})$.  

At this point, we can apply the Theorem~\ref{thm:f-com-det}(4) above
directly and determine the Frobenius complexity of $S_{5,3}$ by
observing that the Perron root of $U$ is $2(15+2\sqrt{30})$.

However, for illustrative purposes let us compute $U^e$. This is
accomplished by diagonalizing $U$. 
%
%

Skipping the details, we get 
\[
U^e= \begin{bmatrix}
(15+4\sqrt{30})y_e  
+(-15+4\sqrt{30})z_e
& 51(y_e-z_e)\\
5(y_e-z_e) 
& (-15+4\sqrt{30})y_e +(15+4\sqrt{30})z_e
\end{bmatrix},
\]
in which
\[
y_e:=\frac{1}{\sqrt{15}} \cdot 2^{-\frac{7}{2}+e} \cdot (15+2\sqrt{30})^e
\quad \text{and} \quad
z_e:=  \frac{1}{\sqrt{15}} \cdot 2^{-\frac{7}{2}+e}\cdot  (15-2\sqrt{30})^e.
\]
Thus, for $e \geq 0$, we obtain
\begin{align*}
X_{e, 1} &= 30( (15+4\sqrt{30})y_{e}
+(-15+4\sqrt{30})z_{e})  
+ 51(y_{e}+z_{e}),\\ 
X_{e,2} &= 150(y_{e}- z_{e}) +
(-15+4\sqrt{30})y_{e}  +
(15+4\sqrt{30})z_{e}.
\end{align*}

Lastly, for $e \geq 2$, we have (cf.~Discussion~\ref{dis:m,r,p})
\begin{equation*}\label{eq:c}
c_e = c_e(T(V_2(K[x_1, \ldots, x_5])))
= \binom{7}{4}X_{e-2, 1} + \binom{6}{4}X_{e-2, 2},
\end{equation*}
which allows us to compute
$c_e(T(V_2(K[x_1, \ldots, x_5])))$ which equals $c_e(\sF(E_{5,3}))$.

Therefore we are led to the following Proposition.
\begin{Prop}
When $p=3$, $\cx_F(S_{5,3}) = \log_3(2(15+2\sqrt{30}))$.
\end{Prop}

At conclusion of the paper, we would like to introduce the definition
of the Frobenius complexity for rings of characteristic zero, which is
motivated by Corollary~\ref{cor:limit} and Theorem~\ref{thm:f-com-det}(3). 
As the definition involves rings that may not be local, we first
extend our Definition~\ref{def-FCX} by defining the Frobenius complexity
of a (not necessarily local) ring $R$ of prime characteristic $p$ as
$\cx_F(R) : = \log_p(\cx(\sC(R)))$. (When $\ringR$ is F-finite
complete local, $\sC(R)$ and $\sF(E(k))$ are opposite as graded rings;
so $\cx(\sC(R)) = \cx(\sF(E(k)))$ and we do have an extension of the definition.)

\begin{Def} \label{FC0}
Let $R$ be a ring (of characteristic zero) such that $R/pR \neq 0$ for
almost all prime number $p$.    
When the limit $\lim_{p \to \infty} \cx_F(R/pR)$ exists, we call it 
\emph{the Frobenius complexity} of $R$. 
\end{Def}

It is natural to ask under what conditions, if any at all, the
Frobenius complexity exists. The case of 
$R = \mathbb Z[X_1,\,\dotsc,\,X_n]/I$ and 
$R = \mathbb Z[[X_1,\,\dotsc,\,X_n]]/I$ are particularly interesting. 
If $R$ is a finitely generated algebra over a field $k$ of characteristic
zero, we could descend $R$ to a finitely generated $A$-algebra
$R_A$ (where $A$ is a subring of $k$ that is finitely generated over
$\mathbb Z$ containing the defining data of $R$) and study the the
Frobenius complexity of $R_A$.

\end{document}